\documentclass[11pt]{article}
\usepackage{graphicx} 
\usepackage{latexsym}
\usepackage{todonotes}
\usepackage{soul} 
\usepackage{amssymb, enumerate}
\usepackage{amsthm} 
\usepackage[leqno]{amsmath} 
\usepackage[right,mathlines]{lineno} 
\usepackage{hyperref}
\usepackage[toc,page]{appendix} 

\usepackage[
backend=biber,
bibstyle=ieee,
citestyle=numeric-comp,
sorting=nyt,
isbn=false
]{biblatex}

\NewBibliographyString{arxiv}
\DefineBibliographyStrings{english}{%
  arxiv = {arXiv}
}

\DeclareSourcemap{
  \maps{
    \map{
    \step[fieldsource=arxivId,
            fieldset=eprint,
            origfieldval]
      \step[fieldsource=arxivId,
            fieldset=eprinttype,
            fieldvalue={arXiv}] 
        \step[fieldset=journal,
        fieldvalue={CoRR}]
    }
  }
}

\renewcommand{\epsilon}{\varepsilon}

\DeclareMathOperator{\vcdim}{\textnormal{vcdim}}
\DeclareMathOperator{\tww}{\textnormal{tww}}

\newtheorem{Theorem}{Theorem}[section]

\newtheorem{Lemma}[Theorem]{Lemma}

\newtheorem{Claim}[Theorem]{Claim}
\newtheorem{Definition}{Definition}[section]

\def\inst#1{$^{#1}$}

\makeatletter
\newcommand{\leqnomode}{\tagsleft@true}
\newcommand{\reqnomode}{\tagsleft@false}
\makeatother

\begin{document}
\fontsize{12}{17}\selectfont

\title{The Twin-Width of Graphs of Bounded VC-Dimension}

\author{Therese Biedl \inst{1}\thanks{We acknowledge the support of the Natural Sciences and Engineering Research Council of Canada (NSERC), [funding reference numbers RGPIN-2020-03912 and RGPIN-2020-03958]. Cette recherche a \'et\'e financ\'ee par le Conseil de recherches en sciences naturelles et en g\'enie du Canada (CRSNG), [num\'eros de r\'ef\'erence RGPIN-2020-03912 et RGPIN-2020-03958]. This project was funded in part by the Government of Ontario. This research was conducted while Spirkl was an Alfred P. Sloan Fellow. This research was undertaken, in part, thanks to funding from the Canada Research Chairs Program.}
\and Taite LaGrange \inst{1}
\and Sophie Spirkl \inst{2}\footnotemark[1]}

\date{\today}
\maketitle

\begin{center}
{\footnotesize
\inst{1} David R. Cheriton School of Computer Science, University of Waterloo, Waterloo, ON, Canada \\
\inst{2} Department of Combinatorics, University of Waterloo, Waterloo, ON, Canada }\\
\end{center}

\begin{abstract}
    In this paper, we investigate which hereditary classes of graphs admit sub-linear (in the number of vertices) bounds on twin-width. By modifying conference graphs, we can show that split, bipartite, and co-bipartite graphs can all have linear twin-width. However, excluding an induced subgraph of each of these three types is equivalent to the class of graphs having bounded VC-dimension, as shown by Bousquet, Lagoutte, Li, Parreau and Thomass\'e. Graphs of bounded VC-dimension can have unbounded twin-width, but whether it can be linear was an open question. In this paper, we first present a tool for obtaining twin-width bounds in general by contracting a graph based on a partition by distinct neighbourhoods. Then, using this tool, we prove that graphs with bounded VC-dimension have twin-width at most sub-linear. We also obtain a separate, tighter upper bound for the class of interval graphs, as well as a lower bound. 
\end{abstract}


\section{Introduction}

For many graph theory problems, there are some classes of graphs in which it is easy to obtain an optimum solution. As such, there are many situations in which we may want to know how far a graph is from one of these `easy' classes, and hence there are numerous `width' parameters to define this distance. Some well-established parameters include tree-width, path-width, and clique-width, each a measure of how structurally distant a graph is from being a tree, path, or clique, respectively.

The usefulness of width parameters cannot be understated, as they allow us to describe the structure of a graph, characterize its complexity with respect to computational problems, and choose or design suitable algorithms to run on it (see, for example, \cite{wd-param}). Choosing the right parameter depends on the problem and all parameters have limitations. Tree-width, for example, is bounded in a smaller range of classes, requiring the graph to be sparse, but can yield more efficient algorithms while clique-width (or, similarly, rank-width) is more general, covering dense families of graphs on top of all those with bounded tree-width.

Recently, twin-width has emerged as being in some sense among the most general of the width parameters. It was introduced by Bonnet, Kim, Thomass\'e, and Watrigant in \cite{tww1} as a generalization of an earlier permutations parameter from Guillemot and Marx \cite{perm}. 
Twin-width is more powerful than both tree-width and clique-width (see Figure 2a in \cite{wd-comp} for a detailed picture of the comparability of common width parameters) since all graphs with bounded tree-width or clique-width also have bounded twin-width. Unlike tree-width and clique-width which are best suited to sparse or dense graphs respectively, graphs of bounded twin-width seem to ``stop at the right place" between the two realms in terms of balancing efficiency with generality \cite{tww2}, so twin-width is useful for a broader range of graphs.

Informally, twin-width measures how structurally different a family of graphs is from the class of cographs (see Section~\ref{sec:prelim} for a formal definition). Cographs, the class of graphs which can be built from a single-vertex graph by repeatedly taking complements and disjoint unions, have nice computational properties and naturally arise in many areas, including recognition algorithms, parallel computing, and biology \cite{cographs}. Quantifying the structural distance between a graph and this class is therefore both practical and natural. However, being a relatively new parameter, much is still unknown about its behaviour, especially in large hereditary classes.

Twin-width is particularly well-suited to problems related to induced subgraphs of graphs. For one, cographs are exactly the class of $P_4$-free graphs and are one of the more significant classes of hereditary graphs \cite{cographs2}. Twin-width is also closed under taking induced subgraphs and complements \cite{tww1}, which makes it a natural candidate for the study of hereditary graph classes in general. Bonnet, Geniet, Kim, Thomass\'e, and Watrigant initially conjectured that every small hereditary class has bounded twin-width \cite{tww2}. This was later disproven in general \cite{tww7}, but is true of small hereditary classes of ordered graphs \cite{tww4}. 

Bounded twin-width is closed under taking FO interpretations and transductions \cite{tww1}. The original algorithmic application of twin-width is in first-order (FO) model-checking, which is FPT-solvable when twin-width is bounded \cite{tww1}. In hereditary classes of ordered graphs, FO-model-checking is FPT-solvable if and only if the class has bounded twin-width \cite{tww4, ord-tww}. In fact, many algorithmic and structural questions have nice answers in classes of graphs with bounded twin-width, including enumeration \cite{tww2}, colouring \cite{tww3, Drei21} and $\chi$-boundedness \cite{Bour23}, and finding minimum or partial dominating sets \cite{tww3, dom-tww}. As mentioned, bounded twin-width encompasses many important graph classes, including graphs with bounded tree-width, bounded clique-width or rank-width, planar graphs, proper minor-closed graph classes, and proper subclasses of permutation graphs \cite{tww1}. 

However, there are limitations to considering only families of graphs with bounded twin-width, as we will see.
While graphs of twin-width zero or one can be recognized in linear time \cite{Ahn25}, little is known even about graphs of twin-width two or three, and deciding if a graph has twin-width four or five is already NP-complete \cite{Berg21}. Most notable to the work presented here, many natural classes of graphs do not have bounded twin-width at all. In fact, some graphs can have linear twin-width relative to the number of vertices, as evidenced by taking an Erd\"{o}s-R\'enyi random graph or a conference graph \cite{rand-tww}, for example. Extending the latter, one can similarly show that each of the three classes of split, bipartite, and co-bipartite graphs have linear twin-width (see Appendix~\ref{appendix} for proof).

This raises a natural question of the relationship between twin-width and Vapnik-Chervonenkis (VC) dimension \cite{vcdim}, a parameter measuring the complexity of vertex neighbourhoods (see Section~\ref{sec:prelim} for a formal definition). A hereditary class of graphs has bounded VC-dimension if and only if it excludes a bipartite graph, a co-bipartite graph, and a split graph \cite{hered-vc, vcdim2}. Since all three classes of bipartite, co-bipartite, and split graphs can have linear twin-width (in the number of vertices), we know that a hereditary class of unbounded VC-dimension contains arbitrarily large graphs of linear twin-width. 

In the opposite direction, Bonnet, Foucaud, Lehtil\"a, and Parreau \cite{bdtww-vc} showed that bounded twin-width implies bounded VC-dimension by arguing about neighbourhood complexity and restrictions on the number of distinct neighbourhoods between fixed-size sets of vertices. For an $n$-vertex graph with twin-width at most $d$, they obtain an essentially tight upper bound of $2^{d+\log d+\Theta(1)}k$ on the VC-dimension.

However, it is known that graphs with bounded VC-dimension can have unbounded twin-width in general: Every bounded twin-width class $\mathcal{C}$ is small in the sense that $\mathcal{C}$ does not contain more than $n!c^n$ graphs with $n$ vertices, for some constant $c$ \cite{tww2}. The class of graphs with VC-dimension at most two contains the class of interval graphs \cite{hered-vc} and these already have unbounded twin-width by the counting argument on the size of the class \cite{tww2}. However, this argument provides little information with respect to precisely how large the twin-width can be for either of these classes. Can it be linear or can we do better? This question is the focus of the current work.

We provide separate answers for both bounded VC-dimension graphs and interval graphs, despite the former containing the latter. We do this for two reasons. First, we can obtain better bounds for interval graphs than for general graphs with bounded VC-dimension, including establishing a lower bound alongside the upper bound. Second, we note the importance of interval graphs in their own right \cite{interval}, justifying this distinction.

Interval graphs have unbounded twin-width but some subclasses, including unit interval graphs, do have bounded twin-width. This was one of the arguments in \cite{tww2} for twin-width seeming to strike the balance between sparse and dense classes.
Moreover, interval graphs are delineated by twin-width \cite{tww8}. Despite these notable aspects of the class and their consistent use as a prime example of graphs with unbounded twin-width, to the best of our knowledge no efforts have been made to sufficiently characterize the behaviour of their twin-width. As such, we devote some attention to obtaining better bounds for this class that improve upon those obtained for general graphs with bounded VC-dimension.

\subsection{Main Results}

We obtain four main results: a general tool for obtaining twin-width bounds, a sub-linear upper bound on the twin-width of graphs with bounded VC-dimension, and both a sub-linear upper bound and a super-logarithmic lower bound on the twin-width of interval graphs.

In Section~\ref{sec:tww}, we first establish a tool that takes as input a graph that is vertex-partitioned into a bounded number of parts, each of which have high similarity between their neighbourhoods, and outputs a contraction sequence witnessing sub-linear twin-width dependent on both the number of parts and the symmetric difference between their neighbourhoods. This is the technique we will use to obtain upper bounds on the twin-width.

Then, in Section~\ref{sec:main}, we find a suitable partition for graphs with bounded VC-dimension to use as input to the aforementioned tool. Feeding this partition into the twin-width tool we designed, we obtain our first result.

\noindent
{\bf Theorem~\ref{thm:vc-tww}}
{\it 	The class of n-vertex graphs with VC-dimension at most $k$ has twin-width in $O(n^{1-\frac{1}{2k+1}})$.}

In particular, Theorem~\ref{thm:vc-tww} implies that a hereditary class $\mathcal{C}$ of graphs admits a sub-linear bound ($o(n)$ in terms of the number of vertices $n$) on twin-width if and only if $\mathcal{C}$ has bounded VC-dimension. The first direction comes from the above, while the second comes from the fact that split, bipartite, and cobipartite graphs are all hereditary classes with linear twin-width (see Appendix~\ref{app:bipartite} for proof) and any class with unbounded VC-dimension contains at least one of these classes, as previously mentioned.

In Section~\ref{sec:interval}, we expand on our previous results to give explicit bounds on the twin-width of interval graphs. We find a partition distinct from that for general graphs with bounded VC-dimension which improves the upper bound on the twin-width. We obtain our second result.

\noindent
{\bf Theorem~\ref{thm:tww-interval-ub}}
{\it 	Any $n$-vertex interval graph has twin-width at most $2\lceil\sqrt{2n}\rceil \in O(\sqrt{n})$.}

We then give a construction for graphs that belong to the intersection of the class of interval graphs and the class of split graphs (graphs whose vertex-set can be partitioned into a clique and a stable set). This construction gives a lower bound on the twin-width of interval graphs.

\noindent
{\bf Theorem~\ref{thm:tww-interval-lb}}
{\it 	The class of $n$-vertex interval graphs has twin-width in $\Omega(n^{1/4})$.}

\section{Preliminaries}\label{sec:prelim}

In this paper, all logarithms are taken base $e$ unless otherwise stated and all graphs are simple and undirected.

\subsection{Graph Theory}
For a graph $G$, let $V(G)$ denote its vertex-set, $E(G)$ its edge-set, and $uv \in E(G)$ the edge between vertices $u,v \in V(G)$. We denote by $N(v)$ the neighbourhood of a vertex $v$ and by $\Delta(G)$ the maximum degree of a vertex in $G$. For a subset of vertices $X \subseteq V(G)$, we denote by $G[X]$ the subgraph induced by $X$ in $G$ and let $|X|$ be the cardinality of $X$. A \emph{stable set} is a set $X$ of vertices where $G[X]$ has no edges.
The following lemma is well-known, for example, as a consequence of Brooks' Theorem \cite{brooks}.

\begin{Lemma}\label{lm:large-IS}
    For any $n$-vertex graph $G$, there exists a stable set of size at least $\frac{n}{1+\Delta(G)}$.
\end{Lemma}

Twin-width will require working with the differences in vertex neighbourhoods, so we define some related basic language here. For two vertices $u,v \in V(G)$, let $N(u) \Delta N(v)$ denote the \emph{symmetric difference} $(N(u) \cup N(v)) \setminus (N(u) \cap N(v))$ between their neighbourhoods.
Consider a vertex-set $S$ and a vertex $v$ (possibly in $S$).   We call $v$ \emph{complete to $S$} if it is adjacent to all vertices of $S$, \emph{anticomplete to $S$} if it is adjacent to none of $S$, and \emph{mixed on $S$} otherwise (it has both a neighbour and a non-neighbour, possibly $v$ itself, in $S$).

\subsection{Twin-Width}

We use the following standard definitions from \cite{tww1}.

A \emph{trigraph} is a graph $G$ along with a partition $(B,R)$ of its edge-set; we usually refer to edges in $B$ as \emph{black} edges and edges in $R$ as \emph{red} edges. We think of red edges as recording `errors' that will arise from `merging' edges and non-edges. A \emph{(vertex) contraction} is an identification of two vertices $u,v$ into a single vertex $w$ where we keep previously black or red edges from $w$ to $N(u) \cap N(v)$ and add red edges from $w$ to $N(u) \Delta N(v)$ (or turn previously black edges red, avoiding multi-edges). In a trigraph, the \emph{red degree} of a vertex $v$ is the number of red edges incident with $v$. 

A \emph{contraction sequence} of an $n$-vertex graph $G$ is a sequence of trigraphs $G = G_n,\dots,G_1 = K_1$ such that $G_t$ is obtained from $G_{t+1}$ by performing a single vertex contraction. We call such a sequence a \emph{d-sequence} if all trigraphs in the sequence have maximum red degree at most $d$. We may also say that a $d$-sequence is a contraction sequence with \emph{width} $d$. The \emph{twin-width} of a graph $G$, denoted $\tww(G)$, is the minimum integer $d$ such that $G$ has a $d$-sequence. 

For a trigraph $G_t$ in a contraction sequence, let $\mathcal{P}_t$ denote the set of \emph{parts} of $G_t$, where a part is a maximal set of vertices that have been contracted to one vertex in $G_t$. For example, if the vertices $u,v \in G_{t+1}$ are contracted together in $G_t$ and neither had previously been contracted with any other vertex, then they will appear in the same part $\{u,v\} \in \mathcal{P}_t$. Hence the trigraph $G_t$ has one vertex for each set in $\mathcal{P}_t$. We refer to parts of size exactly one as \emph{singletons} or \emph{uncontracted vertices}: they have not yet been involved in a contraction with any other vertex. All other vertices of a trigraph $G_t$ are called \emph{contracted vertices}: they represent two or more vertices of $G$.

\subsection{VC-Dimension}\label{sec:prelim-vc}

Originally defined for set systems, VC-dimension essentially characterizes how complex a system is.

\begin{Definition}[Set VC-Dimension \cite{vcdim}]
    For a set $\mathcal{F}$ of subsets of a finite \emph{ground set} $V$, a subset $S$ of $V$ is shattered by $\mathcal{F}$ if for every $A \subseteq S$ there exists $B \in \mathcal{F}$ with $B \cap S = A$. The VC-dimension of $\mathcal{F}$ is the largest cardinality of a subset of $V$ that is shattered by $\mathcal{F}$.
\end{Definition}

Extended to graphs, the VC-dimension of a graph characterizes the complexity of its sets of vertex neighbourhoods.

\begin{Definition}[Graph VC-Dimension \cite{vcdim2}]
    The VC-dimension of a graph $G$, denoted $\vcdim(G)$, is the VC-dimension of the set $\{N(v):v \in V(G)\}$ of neighbourhoods of $V(G)$.
\end{Definition}

Computing the VC-dimension is LOGNP-Complete \cite{comp-vc1, comp-vc2}. Therefore, in this work, we assume that it is given as input along with the graph. We will need a result due to Haussler \cite{same-vc-og}.

\begin{Lemma}[Theorem 1 in \cite{same-vc-og}]\label{lm:haussler}
    Let $\mathcal{F}$ be a set system on an N-element ground set. If the VC-dimension of $\mathcal{F}$ is at most $k$ and $\epsilon = c/N$ for an integer $c$, $1 \leq c \leq N$, then \[\mathcal{M}(\epsilon,\mathcal{F}) \leq e(k+1)\left(\frac{2e}{\epsilon}\right)^k.\]
    Here, $\mathcal{M}(\varepsilon,\mathcal{F})$ denotes the cardinality of the largest subset $\mathcal{F}'$ of $\mathcal{F}$ such that for any $X,Y\in \mathcal{F}'$ we have $|X\Delta Y|\geq \varepsilon N$.
\end{Lemma}

For our use, we apply Lemma~\ref{lm:haussler} to graphs, as follows. 

\begin{Lemma}\label{lm:vc-tww-1}
    Let $G$ be a graph with $\vcdim(G)\leq k$ for some constant $k \in \mathbb{N}$. Consider not-necessarily disjoint subsets $A,B \subseteq V(G)$ with $|A| = a$. Then for all $z \in \{1,\dots,a\}$, either
    \begin{enumerate}[(i)]
        \item $|B| \leq e(k+1)\left( \frac{2e}{z / a}\right)^k$; or
        \item \mbox{there exists a pair of vertices $x,y \in B$ with $x \neq y$ such that $|A \cap (N(x) \Delta N(y))|<z$.}
    \end{enumerate}
\end{Lemma}

This follows from applying Lemma~\ref{lm:haussler} to the set system $\{N(b) \cap A : b \in B\}$ with ground set $A$ (hence $N=a$) and $c=z$.

This tells us that given a graph with bounded VC-dimension, for any two subsets $A$ and $B$ of $G$, either $B$ is small or there exists a pair of vertices in $B$ with neighbourhoods in $A$ that are not substantially different from each other. This is the main tool we will use to help us find vertices with small symmetric difference between their neighbourhoods in Section~\ref{sec:main}, so that we may partition our graph into sets of vertices that make good candidates for contraction while maintaining bounded red degree over the contraction sequence.

\subsection{Interval Graphs}

\begin{Definition}[Interval Graph \cite{interval}]
    For a family $\mathcal{I}$ of closed intervals on the real line, the corresponding interval graph $G_\mathcal{I}$ is the graph with vertex-set $\mathcal{I}$ where two vertices $a, b \in \mathcal{I}$ are adjacent in $G_\mathcal{I}$ if and only if their corresponding intervals have a non-empty intersection.
\end{Definition}

For an interval $I = [a,b]$ with $ a < b$, we say that $I$ has \emph{left endpoint} $a$ and \emph{right endpoint} $b$. Without loss of generality, we may assume that all endpoints are integers. 
In fact, if $G$ is an $n$-vertex interval graph, we may assume that $G$ corresponds to a family $\mathcal{I}$ of intervals such that all intervals in $\mathcal{I}$ have distinct endpoints in $\{1, \dots, 2n\}$.

\section{The Twin-Width Tool}\label{sec:tww}
Our first result is a tool for obtaining twin-width bounds in general.

Theorem~\ref{thm:secret-tww} below tells us that to show that an input graph $G$ has sub-linear twin-width in the size of $G$, we need only find a partition of $G$ into a bounded number of sets in a sequence such that each but at most one is mixed on a bounded number of vertices later in the sequence, and the one that may be mixed on more has bounded size. The tool gives a contraction sequence witnessing the twin-width as a function of the number of sets, the number of vertices they are mixed on, and the size of the single set that may be mixed on more. This is the technique we will use to obtain upper bounds on the twin-width of graphs of bounded VC-dimension in Section~\ref{sec:main} and of interval graphs in Section~\ref{sec:interval}.

\begin{Theorem}\label{thm:secret-tww}
    Let $\alpha, \beta, \gamma \in \mathbb{R}_{\geq 0}$. Let $G$ be a graph with $V(G)$ partitioned into a disjoint sequence of sets $X_1, \dots, X_t, X_{t+1}$ that satisfy
    \begin{itemize}
        \item $t \leq \alpha$;
        \item $|X_{t+1}| \leq \beta$; and
        \item for each $X_i$ with $i \in \{1,\dots,t\}$, letting \[Y_i = \{ v \in X_{i} \cup \dots \cup X_{t+1}: v\text{ is mixed on } X_i \},\] we have that $|Y_i| \leq \gamma$.
    \end{itemize} 
    Then $\tww(G) \leq \alpha + \max\{\gamma, \beta\}$.
\end{Theorem}
\begin{proof}
    Let $G$ be a graph with its vertices partitioned into sets $X_1, \dots, X_t, X_{t+1}$ satisfying the assumptions of Theorem~\ref{thm:secret-tww}. Then a contraction sequence witnessing the twin-width of $G$ is as follows.

    We describe the sequence in three phases. In the first phase, for each $X_i$ in order from 1 to $t$, contract $X_i$ down to a single vertex by arbitrarily choosing two starting vertices to contract to one part $x_i$, then contracting each of the remaining uncontracted vertices of the set one by one into $x_i$. This maintains that $x_i$ is an endpoint for every red edge with an endpoint in $X_i$. The order of vertices contracted does not matter.
    
    During this sequence, while a set $X_i$ is being contracted to a single vertex $x_i$ for some index $i$, $x_i$ may have red edges to the previously contracted vertices $x_j$ for $j < i$ (when they exist), and to the set $Y_i$ of vertices that were mixed on $X_i$. Note that $x_i$ has no red edges to $X_j \setminus  Y_i$ for $j > i$ by definition of $Y_i$. Hence the red degree of $x_i$ is at most $t + |Y_i| \leq \alpha + \gamma$. The red degree of previously contracted $x_j$'s for all $j < i$ remains at most $t + |Y_j| \leq \alpha + \gamma$ since they may have red edges only to $Y_j$ and to the vertices $x_1, \dots, x_i$. For all uncontracted vertices, they have red edges only to the contracted vertices and so their red degree is at most $\alpha$. 
    
    With the sets $X_1, \dots, X_t$ contracted, we focus the second phase of the sequence on the final set $X_{t+1}$. The process of contracting $X_{t+1}$ is the same, repeatedly contracting to a single vertex $x_{t+1}$, but the analysis of the red degree is slightly different. While contracting $X_{t+1}$, the vertex $x_{t+1}$ may have red edges to the previous $x_i$'s for $i \in \{1, \dots, t\}$ and to the not yet contracted vertices of $X_{t+1}$. Hence the red degree of $x_{t+1}$ is at most $t + |X_{t+1}| \leq \alpha + \beta$. The red degree of previously contracted $x_j$ for all $j < t+1$ remains at most $t + |Y_j| \leq \alpha + \gamma$ because they may have red edges only to $Y_j$ and to the vertices in $\{x_1, \dots, x_{t+1}\} \setminus \{x_j\}$. 
    
    For the third and final phase, let $P$ be the vertex-set after the end of the second phase (so $P = \{x_1, \dots, x_{t+1}\}$) and recall that $\alpha\geq t$. We may now finish the sequence by contracting $P$ to a single vertex without creating a vertex of red degree greater than $\alpha$.
    
    Throughout the contraction sequence, we have maintained that all red degrees are at most $\alpha + \max\{\gamma,\beta\}$.
\end{proof}

\section{Graphs of Bounded VC-Dimension}\label{sec:main}

We first need to structure graphs with bounded VC-dimension into a vertex-partition satisfying the constraints of Theorem~\ref{thm:secret-tww} above. Building on Lemma~\ref{lm:vc-tww-1} from Section~\ref{sec:prelim-vc}, we show that there always exists some reasonably-sized set of vertices such that the cumulative difference between their neighbourhoods is not too large. Some of this is similar to the ideas of Theorem 3 in \cite{same-vc}, which uses bounded VC-dimension to partition the graph into a bounded number of sets while controlling edge densities between sets. With the tool of Lemma~\ref{lm:vc-tww-2} in place, we will then be able to partition the graph into the sets needed to apply the technique of Theorem~\ref{thm:secret-tww} and obtain an upper bound on the twin-width.

\begin{Lemma}\label{lm:vc-tww-2}
    Let $k \in \mathbb{N}$ and $\delta_k, n_0 \in \mathbb{R}_{\geq 0}$ such that $0 < \delta_k < \frac{1}{2k+1}$ and $n_0 > \left(e(k+1)(4e)^k\right)^{\frac{1}{1-\delta_k(2k+1)}}$.
    
    For every $n$-vertex graph $G$ with $\vcdim(G) \leq k$ and $n > n_0$, there exists $X \subseteq V(G)$ with $|X| \geq n^{\delta_k}$ such that $Y = \{v \in V(G) : \text{mixed on X}\}$ satisfies $|Y| \leq n^{1-\delta_k}$.
\end{Lemma}

\begin{proof}
    We used $\delta_k$ in the lemma statement to indicate that this is a parameter dependent on the VC-dimension $k$, but for ease of writing the rest of the proof will use $\delta$ in its place.
    
    For a vertex $v \in V(G)$, define $S(v) = \{u : | N(v) \Delta N(u) | \leq n^{1-2\delta}\}$. That is, $S(v)$ is the set of vertices whose neighbourhoods are not substantially different from the neighbourhood of $v$. Note that $S(v)$ includes $v$ itself.
    
    {\bf Case 1:} There exists a vertex $v \in V(G)$ with $|S(v)| \geq n^{\delta}$. Pick a set $X$ of vertices with $\{v\} \subseteq X \subseteq S(v)$ and $|X| = \lceil n^{\delta} \rceil$. We claim that the set $Y$ of vertices mixed on $X$ is the union of the sets of vertices in the symmetric difference between $N(v)$ and $N(u)$ over all choices of $u \in X$. Clearly $Y$ is a superset of $\bigcup_{u\in X} N(v)\Delta N(u)$ by definition of `mixed'. To see that it is also a subset, let $x' \in Y$ be mixed on two vertices $u,w \in X$. Assume that $x'$ is adjacent to $u$ and non-adjacent to $w$. If $v$ is adjacent to $x'$, then $x'$ is in $N(v) \Delta N(w)$. If $v$ is non-adjacent to $x'$, then $x'$ is in $N(v) \Delta N(u)$.
    
    From the definition of $S(v)$, we know that each set $N(v) \Delta N(u)$ consists of at most $n^{1-2\delta}$ vertices. So $|Y| \leq (|X|-1) \cdot n^{1-2\delta} = \left(\lceil n^{\delta} \rceil-1\right) \cdot n^{1-2\delta} \leq n^{1-2\delta+\delta} = n^{1-\delta}$ and the lemma holds in this case.
    
    {\bf Case 2:} For every vertex $v \in V(G)$, we have $|S(v)| < n^{\delta}$. We will show that this is impossible. To this end, construct an auxiliary graph $H$ on the same vertex-set and make two vertices adjacent if and only if there is a small symmetric difference between their neighbourhoods. Formally, we take $V(H) = V(G)$ and, for distinct vertices $u,v \in V(H)$, the edge $uv$ is in $E(H)$ if and only if $|N(v) \Delta N(u)| \leq n^{1-2\delta}$. For a vertex $v$, let $N_H(v)$ denote the neighbourhood of $v$ in $H$ and $deg_H(v)$ the degree of $v$ in $H$. Then $\{v\} \cup N_H(v) = S(v)$ and $deg_H(v) \leq  n^{\delta}  - 1$. So $\Delta(H) \leq n^\delta-1$.
    
    By Lemma~\ref{lm:large-IS} on the minimum size of a stable set, there exists a set of vertices $B \subseteq V(H)$ that is stable in $H$ with $|B| \geq \frac{n}{ n^{\delta} } = n^{1-\delta}$. For this set $B$ and the set $A = V(G)$, we may now apply Lemma~\ref{lm:vc-tww-1} from Section~\ref{sec:prelim-vc}. We obtain that for all $z \in \{1,\dots,n\}$ either 
    \begin{align*}
\text{(1) } & ~n^{1-\delta} \leq |B| \leq e(k+1)\left( \frac{2e}{z / n}\right)^k \text{ or} \\
\text{(2) } & ~n^{1-2\delta} < z.
\end{align*}

    Observe that for $z = \lceil\frac{n^{1-2\delta}}{2}\rceil$, inequality (2) does not hold. This works because 

    \[n^{1-2\delta} \geq n_0^{1-2\delta} \geq (4e)^{k(1-\frac{2}{2k+1})}\geq (4e)^\frac{2}{3} \geq 2\]

    which implies that $z \geq 1$ and so $z\leq \tfrac{ n^{1-2\delta+1}}{2} \leq \tfrac{2n^{1-\delta}-1}{2} < n^{1-\delta}$, contradicting (2).

So for $z=\lceil \frac{n^{1-2\delta}}{2} \rceil$ inequality (1) must hold, and  we obtain

    \[n^{1-\delta} \leq e\left(k+1\right)\left( \frac{2e}{\lceil n^{1-2\delta}/{2}\rceil n}\right)^k \leq e\left( k+1\right)\left(\frac{2e}{n^{1-2\delta}/2n}\right)^k = e\left(k+1\right)\left(4en^{2\delta}\right)^k.\]
    
We can rearrange this to obtain
    \[n^{1-\delta-2\delta k} \leq e(k+1)(4e)^k.\]

Isolating for $n$, we have that \[n < \left(e(k+1)(4e)^k\right)^{\frac{1}{1-\delta(2k+1)}}\] which is a contradiction to the assumption that $n> n_0$. 
    
    Therefore we are always in the first case, and the lemma holds.
\end{proof}

We now have all of the necessary tools to show that bounded VC-dimension implies sub-linear twin-width. We will use the previous result to obtain the partition that will serve as input to Theorem~\ref{thm:secret-tww}, which will then give us a bound on the twin-width.

\begin{Theorem}\label{thm:vc-tww}
    Let $k \in \mathbb{N}$ and let $\delta_k, n_0$ be as in Lemma~\ref{lm:vc-tww-2}. For an $n$-vertex graph $G$ with $n^{1-\delta_k} \geq n_0$, if $\vcdim(G) \leq k$, then $\tww(G) \in  O(n^{1-\delta_k})$.
\end{Theorem}
\begin{proof}
    We again write $\delta$, rather than $\delta_k$, to simplify notation.
    Let $X_0$ be the empty set. Define $X_1,\dots,X_t$ iteratively as follows, where $i$ is index of the current iteration.

    \begin{enumerate}
        \item\label{stp:delete} Consider $G' = G\backslash (\bigcup_{j<i}X_j)$. Note that $\vcdim(G') \leq \vcdim(G) \leq k$.
        \item\label{stp:stop} Set $n' = |G'|$.  If $n' \leq n^{1-\delta}$, then stop and set $t=i-1$ and $X_{t+1}= V(G')$.
        \item\label{stp:lemma} We now know that $n' > n^{1-\delta}$. Apply Lemma~\ref{lm:vc-tww-2} to $G'$ to obtain a set $X_i$ with $|X_i| \geq (n')^{\delta}$ such that the set $Y_i$ of vertices mixed on $X_i$ satisfies $|Y_i| \leq (n')^{1-\delta}$. 
    \end{enumerate}
    
    We argue that the partition $G = \{X_1,\dots,X_t, X_{t+1}\}$ obtained above satisfies the constraints of Theorem~\ref{thm:secret-tww}.

    First, we analyze the number of iterations executed by the algorithm. We group our analysis into phases defined by the size of $G'$. 

For $0\leq j\leq \lceil \log_2 n \rceil$, let phase $j$ 
encompass all iterations for which $\frac{n}{2^{j+1}} < n' \leq \frac{n}{2^j}$. 
If iteration $i$ falls in Phase $j$, then in Step~\ref{stp:lemma} we have $|X_i| \geq (n')^{\delta} \geq \left(\frac{n}{2^{j+1}}\right)^{\delta}$. 
So we delete at least $\left(\frac{n}{2^{j+1}}\right)^{\delta}$ vertices in each iteration of Phase $j$, implying that it has at most $\left\lceil \frac{n}{2^{j+1}} / \left(\frac{n}{2^{j+1}}\right)^{\delta} \right\rceil \leq 2\left(\frac{n}{2^{j+1}}\right)^{1-\delta}$ iterations. 
Then the total number of iterations, and thus the total number $t$ of sets $X_i$ obtained, is at most

    \[\sum_{j=0}^{\lceil\log_2n\rceil} 2\left( \frac{n}{2^{j+1}} \right)^{1-\delta} 
  2n^{1- \delta}\sum_{j=0}^{\lceil\log_2 n\rceil} \left(\left(\frac{1}{2} \right)^{1-\delta}\right)^{j+1}
  \in O(n^{1-\delta}) \]
   and we set $\alpha$ to be this number of sets. 
    
    Note that $n^{1-\delta} \geq n_0$ by assumption. So the set $X_{t+1}$ has size at most $n^{1-\delta}$, which we take as the parameter $\beta$.

    Each set $Y_i$ of vertices mixed on $X_i$ has size at most $(n')^{1-\delta} \leq n^{1-\delta}$ by choice of $X_i$ and we set $\gamma=n^{1-\delta}$.
    
    We now have $X_1,\dots, X_{t+1}$ covering all of $G$ and satisfying the constraints of Theorem~\ref{thm:secret-tww} for parameters 
    $\alpha,\beta,\gamma\in O(n^{1-\delta})$.
   It follows that $\tww(G) \in O(n^{1-\delta})$.
\end{proof}

The above result gives us that all graphs with small VC-dimension have twin-width at most sub-linear. As a corollary, in Table~\ref{table:vc-tww} we list a few select classes of graphs with bounded VC-dimension that we find interesting and state explicitly the bound obtained from Theorem~\ref{thm:vc-tww}. While planar graphs are known to have twin-width at most 8 \cite{planartww8}, we list them here to provide a frame of reference.

\begin{table}[h!]
\begin{center}
    \begin{tabular}{c |c |c }
        \emph{Classes} & \emph{vcdim} & \emph{tww} \\
        \hline
    
        Interval \cite{hered-vc} & $\leq 2$ & $\leq n^{\frac{4}{5}+\varepsilon}$\\
    
        Unit Disk \cite{hered-vc}, Permutation \cite{permutationvc} & $\leq 3$ & $\leq n^{\frac{6}{7}+\varepsilon}$\\
    
        Disk \cite{diskvc}, Line \cite{linevc}, Planar\cite{diskvc} & $\leq 4$ & $\leq n^{\frac{8}{9}+\varepsilon}$\\
    \end{tabular} 
\end{center}
\caption{Classes of bounded VC-dimension and the upper bounds on their twin-width, which hold for arbitrarily small $\varepsilon$.}\label{table:vc-tww}
\end{table}

\section{Interval Graphs}\label{sec:interval}

We now proceed to give better upper bounds on the twin-width of interval graphs than those obtained from Theorem~\ref{thm:vc-tww}, and then also give lower bounds. The lower bound is constructive, while the upper bound leverages the ordering of neighbourhoods to find a suitable partition for Theorem~\ref{thm:secret-tww} from Section~\ref{sec:tww}.

\begin{Theorem}\label{thm:tww-interval-ub}
    For an n-vertex interval graph $G_\mathcal{I}$, $\tww(G_\mathcal{I}) \in O(\sqrt{n})$.
\end{Theorem}

\begin{proof}
Let $\mathcal{I}$ be a family of $n$ intervals.
Without loss of generality, assume that no two intervals have an endpoint in common and that all endpoints are integers
in the range $\{1, \dots, 2n\}$. 
For the rest of this proof, we refer to the vertices of $G_\mathcal{I}$ and the intervals of $\mathcal{I}$ interchangeably.

Define points on the real line $x_i = i\cdot \sqrt{2n}$ for $i \in \{1,\dots, \lceil\sqrt{2n}\rceil\}$ and sets $S_i = \{I \in \mathcal{I}: I = [a,b], i = \min\{j : x_j \geq b\}\}$. Let $x_0=0$. Observe that all intervals in $S_i$ have their right endpoint in the range $R_i = (x_{i-1}, x_i]$, which is a range of length $\sqrt{2n}$.

We want to use the sets $S_i$ as the partition of $G_\mathcal{I}$ for Theorem~\ref{thm:secret-tww}. Then, to satisfy the constraints on this partition, we need two conditions.
\begin{enumerate}[(1.)]
    \item \label{1} The sets $S_i$ are pairwise disjoint and $\bigcup S_i = V(G_\mathcal{I})$.
    \item \label{3} For each $S_i$, there is a bounded number of vertices that are mixed on $S_i$ in the union of $S_i$ and of all later $S_j$'s for all $j \geq i$.
\end{enumerate}

The first condition is immediate. We note that the number of sets $S_i$ is $\lceil\sqrt{2n}\rceil$, which we will take as the parameter $\alpha$ for Theorem~\ref{thm:secret-tww}. We will now show that we have (\ref{3}.) by analyzing which vertices may be mixed on a given $S_i$.

We claim that an interval $l = [a,b]$ in $S_j$  (for $j \geq i$) is mixed on $S_i$ if and only if $l$ has an endpoint in $R_i$. Clearly all intervals in $S_i$ have an endpoint in $R_i$, so we only need to show this for $j>i$ where we know that $b > x_i$ is not in $R_i$. So assume that $a \notin R_i$. If $a > x_i$, then $l$ is disjoint from all intervals in $S_i$, and hence $l$ is anti-complete to $S_i$. Otherwise, if $a \leq x_{i-1}$, then $l$ intersects all intervals in $S_i$, and hence $l$ is complete to $S_i$. 

So the only intervals that are mixed on $S_i$ are those with at least one endpoint in $R_i$. By the assumptions that no two intervals share an endpoint in common and that all endpoints are integers, there are at most $\lceil\sqrt{2n}\rceil$ intervals with an endpoint in $R_i$. This bound of $\lceil\sqrt{2n}\rceil$ is exactly what we need for (\ref{3}.), which we take as the parameter $\gamma$ for Theorem~\ref{thm:secret-tww}.

Note that since all sets $S_i$ have such a bound on the number of intervals which are mixed on them, we do not need the set $X_{t+1}$ from Theorem~\ref{thm:secret-tww}. So we may treat this set as empty and take the parameter $\beta$ to be zero.

Then applying Theorem~\ref{thm:secret-tww} with the partition $X_i=S_i$ for all $i \in \{1,\dots,t\}$ and $t = \lceil\sqrt{2n}\rceil$, and $X_{t+1} = \emptyset$ and with the parameters $\alpha = \lceil\sqrt{2n}\rceil$, $\gamma = \lceil\sqrt{2n}\rceil$, and $\beta = 0$, we obtain that $\tww(G_\mathcal{I}) \leq 2\lceil\sqrt{2n}\rceil$ which is in $O(\sqrt{n})$.
\end{proof}

\begin{Theorem}\label{thm:tww-interval-lb}
Let $r \in \mathbb{N}$ with $r \geq 2$, 
and let $n = 3r^2-2$. Then there exists an $n$-vertex interval graph with twin-width at least $(\tfrac{n+2}{12})^{1/4}-1$.
\end{Theorem}

\begin{proof}
We construct an interval graph $G_\mathcal{I}$ on a family $\mathcal{I}$ of $n$ intervals. 
We refer to the intervals in $\mathcal{I}$ and the vertices in $G_\mathcal{I}$ interchangeably throughout this proof, and we will drop $G_\mathcal{I}$ as subscript of the neighbourhood $N_{G_\mathcal{I}}(v)$ of a vertex since we will study no other graph. 
The idea is to make $G_\mathcal{I}$ a split graph such that any clique-vertex is adjacent to many stable-set vertices, yet for any two clique-vertices, their neighbourhoods are quite different.

Figure~\ref{fig:tww-interval-lb} illustrates the following construction. 
First, create a clique $\cal K$ of size $k = r^2$. The key idea here is that any two vertices in $\cal K$ whose left endpoints are close together will have right endpoints that are far apart.
Formally, for all indices $i,j \in \{0,\dots,r-1\}$, we define the interval $K_{i,j} = [ri+j, k+rj+i]$ and let $\mathcal{K} = \{K_{i,j} : i,j \in \{0,\dots, r-1\}\}$. Observe that all intervals in $\cal K$ intersect the point $k-\frac{1}{2}$, so they indeed form a clique.

Now construct a stable set $\mathcal{S}$ on $2k-2$ vertices as follows. For each index $i \in \{0, \dots,2k-2\}$, define the interval $S_i = [i+\frac{1}{3}, i+\frac{2}{3}]$. Let ${\cal S} = \{S_i : i \in \{0,\dots,2k-2\}$ and $i \neq k-1\}$ and let ${\cal S}_L = \{S_i : i \in \{0,\dots,k-2\}\}$ and ${\cal S}_R = S \setminus S_L$. Note that $|\mathcal{S}| = 2k-2$ as desired.

\begin{figure}[ht]
\hspace*{\fill}
\includegraphics[width=\linewidth]{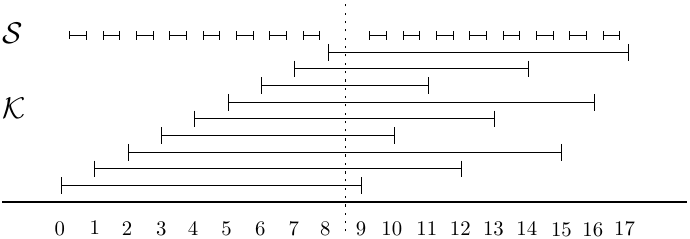}
\hspace*{\fill}
\caption{The constructed graph for 
Theorem~\ref{thm:tww-interval-lb}; here $r=3$ and $k=9$. From bottom to top, the intervals in $\mathcal{K}$ are $K_{0, 0}, K_{0, 1}, K_{0, 2}, K_{1,0}, K_{1,1}, K_{1,2}, K_{2, 0}, K_{2,1}, K_{2,2}.$
}
\label{fig:tww-interval-lb}
\end{figure}

This completes the construction of $G_\mathcal{I}$.   We first prove two useful properties of $G_\mathcal{I}$.
Our first claim concerns the neighbourhood of any clique-vertex.

\begin{Claim}
{\label{eq:k-nbrs} Each interval $K_{i,j}$ in $\mathcal{K}$ satisfies $|N(K_{i,j})\cap S| \geq 2r-2$.}
\end{Claim}
\begin{proof}
There are three cases. In the first case, interval $K_{i,j}$ starts far enough to the left such that it intersects the rightmost $2r-2$ intervals of $\mathcal{S}_L$, i.e., $S_{k-2}$ and the $2r-3$ intervals of $\mathcal{S}_L$ to its left.   This holds if $ri+j\leq k-2r+3=r^2-2r+1$, which is the case if either $i\leq r-3$ or $i=r-2$ and $j\leq 1$.   
Similarly, if $j \geq 2$ or both $j=1$ and $i\geq r-2$, then $K_{i,j}$ intersects the leftmost $2r-2$ intervals of $\mathcal{S}_R$.
This leaves the case where $i=r-1$ and $j=0$. Here $K_{i,j}$ intersects the $r-1$ rightmost intervals of $\mathcal{S}_L$ and the $r-1$ leftmost intervals of $\mathcal{S}_R$, i.e., $S_{(r-1)r},\dots, S_{(r-1)r+r-2}$ and $S_k, \dots, S_{k+r-2}$. In all three cases, $K_{i,j}$ intersects at least $2r-2$ intervals as desired.
\end{proof}

The next claim formalizes the earlier statement that for any two clique-vertices the neighbourhoods are quite distinct.

\begin{Claim}
{\label{eq:x-size} For any two intervals $K_{i,j},K_{i',j'}\in \mathcal{K}$, we have  $|N(K_{i,j}) \Delta N(K_{i',j'})| \geq r$.}
\end{Claim}
\begin{proof}
Write $L=ir+j, L'=i'r+j', R=jr+i$ and $R'=j'r+i'$ for the left and right endpoints of these intervals.   We are done if
$|L'-L|\geq r$ or $|R'-R|\geq r$, because then the (at least $r$) intervals of $\mathcal{S}$ that fall into the gap
between the endpoints belong to one of the neighbourhoods, but not the other.   We now show
that this is nearly always the case, and the one exception can be handled as well.

\medskip\noindent{Case 1:} $|i-i'|\geq 2$ or $|j-j'|\geq 2$.   Let us assume that $|i-i'|\geq 2$, the other case is done similarly by considering $R,R'$.
Up to renaming of the two intervals, we have $i'\geq i+2$, and therefore $L'= i'r+j'\geq (i+2)r > ir+r+j=L+r$.

\medskip\noindent{Case 2:} $i=i'$ or $j=j'$.   Let us assume that $j=j'$, the other case is done similarly by considering $R,R'$.
Since the intervals are distinct, we have $i\neq i'$, hence up to renaming of the two intervals we have $i'\geq i+1$ and 
$L' = i'r+j' \geq (i+1)r+j = L+r$.

\medskip\noindent{Case 3:} $|i-i'|=1$ and $|j-j'|=1$. Up to renaming we have $i'=i+1$.   If furthermore $j'=j+1$, then   we are done since
$L'= i'r+j'=(i+1)r + j+1= L+r+1$.   This leaves as final case $j'=j-1$, where similar computations show $L'=L+r-1$ and $R=R'+r-1$.
Hence in each of $\mathcal{S}_L$ and $\mathcal{S}_R$ there are $r-1$ intervals in the symmetric difference, and in total there
are $2r-2\geq r$.
\end{proof}

We now show the lower bound on the twin-width of $G_\mathcal{I}$.
Consider a contraction sequence $G_\mathcal{I} = G_n, \dots, G_1$ of width $w$. We will show that $w+1\geq \sqrt{r/2}$, which with $r=\sqrt{k}=\sqrt{\tfrac{n+2}{3}}$ gives the result after appropriate reformulations.

Let $q$ be maximal such that $G_q$ has a part $P$ that is not a singleton and contains a clique-vertex $K_{i,j}$. Note that $q < n$, since all parts of $G_n$ are singletons.   Also in $G_{q+1}$ vertex $K_{i,j}$ was a singleton by choice of $q$, so we obtained part $P$ of $G_q$ by merging some part $P'$ of $G_{q+1}$ with $K_{i,j}$. We claim that no part of $G_{q+1}$ can have been very big.

\begin{Claim}
{\label{eq:part-size}Any part $Q$ of $G_{q+1}$ has size $|Q| \leq 2w+2$.}
\end{Claim}
\begin{proof}
We may assume that $Q \subseteq \mathcal{S}$, for otherwise $Q \cap K \neq \emptyset$ and $|Q|= 1 \leq 2w-2$ by choice of $q$. Up to symmetry between $\mathcal{S}_L$ and $\mathcal{S}_R$, we may assume that $|Q \cap \mathcal{S}_L| \geq \frac{|Q|}{2}$. Let $l = \min\{l':S_{l'} \in Q \cap \mathcal{S}_L\}$ and $m = \max\{l':S_{l'} \in Q \cap \mathcal{S}_L\}$ so that $S_l$ and $S_m$ are the leftmost and rightmost intervals of $\mathcal{S}_L$ in $Q$, respectively. Since $S_l$ has right endpoint $l + \frac{2}{3}$ and $S_m$ has left endpoint $m + \frac{1}{3}$, both of which are less than $k$, it follows that all intervals in $\mathcal{K}$ with left endpoint between $l+1$ and $m$ are mixed on $Q$. Since $Q\subseteq \mathcal{S}$, therefore $Q$ has red edges to every vertex in the set $\{I^x:x \in \{l+1,\dots,m\}\}$, where $I^x$ is the interval of $\mathcal{K}$ whose left endpoint is $x$.   (Note that this exists for all $x\in \{0,\dots,k-1\}$.) Then $Q$ has red degree at least $m-l \geq |Q\cap \mathcal{S}_L|-1 \geq \frac{|Q|}{2}-1$. Since the red degree of $Q$ is at most $w$, the claim follows.
\end{proof}

Now we study the part $P'=P\setminus K_{i,j}$ of $G_{q+1}$ to argue a lower bound on $w$.
As first case, assume that $P' $ contains vertices of $\mathcal{S}$. Let $\mathcal{S}' = (N(K_{i,j})\cap \mathcal{S})\setminus P'$. By 
Claims~%
\eqref{eq:k-nbrs}
and 
\eqref{eq:part-size}%
, we have $|\mathcal{S}'|\geq 2r-2-(2w+2)$. Again by Claim \eqref{eq:part-size}, at least $\tfrac{|\mathcal{S}'|}{2w+2}\geq \frac{2r-2}{2w+2}-1$ parts in $G_q$ contain a vertex in $\mathcal{S}'$.   Let $\hat{P}$ be one such part,  say it contains $\hat{S}\in \mathcal{S}'$.  By definition of $\mathcal{S}'$ we have $\hat{P}\neq P'$.   Therefore $\hat{P}$ is mixed on $P$, because $\hat{S}$ is adjacent to $K_{i,j}$ by definition of $\mathcal{S}'$, and $\hat{S}$ is not adjacent to the vertices of $\mathcal{S}$ in $P'$.
It follows that, in $G_q$, part $P$ has red edges to at least $\tfrac{2r-2}{2w+2}-1$ other parts, hence
$w \geq \frac{r-1}{w+1}-1$, or  $(w+1)^2\geq r-1 \geq \tfrac{r}{2}$ as desired.

\medskip
Now we consider the other case where $P'\subseteq K$, which by our choice of $q$ means that $P'$
is a singleton, say $P' = \{K_{i',j'}\}$.  Consider the symmetric difference $X=N(K_{i,j}) \Delta N(K_{i',j'})$,  which by Claim~\ref{eq:x-size} has size $|X|\geq r$.
In $G_q$, part $P=\{K_{i,j},K_{i',j'}\}$ 
has a red edge to every part of $G_q$ containing a vertex of the set $X$, which is at least $\frac{r}{2w+2}$ parts by Claim \eqref{eq:part-size}. 
Hence $w\geq \frac{r}{2w+2}$, which implies $\tfrac{r}{2}\leq w(w+1)<(w+1)^2$ as desired.
\end{proof}


\newpage

\printbibliography

\begin{appendices}

\section{The Twin-width of Split, Bipartite, and Cobipartite Graphs}\label{appendix}

We need two additional definitions in this section. A \emph{conference graph} of order $n$ is a graph $G$ where every vertex $v\in G$ has degree $\frac{n-1}{2}$ and for each pair $u,v \in G$, if $u$ and $v$ are adjacent then $|N(u) \cap N(v)| = \frac{n-5}{4}$ and if $u$ and $v$ are non-adjacent then $|N(u) \cap N(v)| = \frac{n-1}{4}$. The \emph{tensor product} $G \times H$ of two labelled graphs $G$ and $H$ is the graph whose vertex set is the Cartesian product of $V(G)$ and $V(H)$ and two vertices $(u,v),(u',v')\in G \times H$ are adjacent if and only if $uu'$ is an edge in $G$ and $vv'$ is an edge in $H$.

\begin{Lemma}\label{app:bipartite}
    Bipartite, split, and cobipartite graphs all have twin-width in $\Theta(n)$.
\end{Lemma}
\begin{proof}

    By Ahn, Hendrey, Kim, and Ohm \cite{Ahn_2022}, there exists an infinite class $\mathcal{C}$ of conference graphs of order $n$ with twin-width at least $\frac{n-1}{2}$.
    Let $G$ be a graph in $\mathcal{C}$. By definition, for any graph $G \in \mathcal{C}$, it holds that $|N(v)| = \frac{n-1}{2}$ for all $v \in G$. Also, for any pair of distinct vertices $u,v \in G$, it holds that $|N(u) \Delta N(v) \setminus \{u,v\}| = \frac{n-1}{2}$.

    Take $H$ to be the tensor product $G \times P_2$ of $G$ with the path on two vertices whose vertices are labelled $'$ and $''$. The tensor product defines $V(H) = \{v', v'' : v\in V(G)\}$ and $E(H) = \{u'v'', u''v' : uv \in E(G)\}$. Observe that the sets $\{v' \in G' | v \in G\}$ and $\{v'' \in G' | v \in G\}$ are stable and so this construction is bipartite.

    Note that it is sufficient to show that $H$ preserves the property that, for any two vertices $u,v \in H$, we have $|N(u) \Delta N(v)| \geq \frac{n-1}{2}$ as this ensures linear red degree from the first contraction.

    First, for any vertex $v \in H$, this construction maintains $|N(v)| = \frac{n-1}{2}$. For any two vertices $u',v'' \in H$ on opposite sides of the bipartition, $|N(u') \Delta N(v'')| = |N(u') \cup N(v'')| \geq 2\frac{n-1}{2}-2 \geq n-3$. For any two vertices $u',v' \in H$ on the same side of the bipartition, $|N(u') \Delta N(v')| = |\{w'' : w \in N(u) \Delta N(v)\}| \geq \frac{n-1}{2}$.

    Note that twin-width is self-complementary: a graph and its complement have the same twin-width \cite{tww1}. Thus, the lemma statement holds for cobipartite graphs by taking the complement of the construction above.

    To see that this also holds for split graphs, consider the case when one side of the bipartition is made into a clique by adding all possible edges. For any two vertices $u', v'$ on the same side of the bipartition, the symmetric difference of their neighbourhoods has not changed (except that it now includes $u', v'$ if $u', v'$ are on the clique side). For any two vertices $u',v'' \in H$ on opposite sides of the bipartition such that $u'$ belongs to the clique and $v''$ to the stable set, $|N(u') \Delta N(v'')| = |N(u') \setminus \{v''\}| \geq \frac{n-1}{2}-1$. 
\end{proof}

\end{appendices}
\end{document}